\documentclass[12pt]{amsart}

\usepackage[margin=3cm]{geometry}

\usepackage{amsfonts,amsmath,amsthm,amssymb,graphicx,epsfig,latexsym}
\usepackage{cleveref}
\usepackage[numbers,sort&compress]{natbib}

\newtheorem{theorem}{Theorem}
\newtheorem{prop}[theorem]{Proposition}

\newtheorem{cor}[theorem]{Corollary}

\newtheorem*{problem*}{Open Problem}

\crefname{theorem}{Theorem}{Theorems}
\crefname{lemma}{Lemma}{Lemmas}
\crefname{prop}{Proposition}{Propositions}
\crefname{cor}{Corollary}{Corollaries}
\crefname{section}{Section}{Sections}
\crefname{figure}{Figure}{Figures}
\crefname{table}{Table}{Tables}

\newcommand{\R}{{\mathbb R}}
\newcommand{\Z}{{\mathbb Z}}
\newcommand{\N}{{\mathbb N}}

\newcommand{\tM}{\widetilde{M}}

\newcommand{\LA}{L_A}
\newcommand{\LB}{L_B}

\newcommand{\df}{\textbf}

\DeclareMathOperator{\mex}{mex}

\usepackage{color}

\title[Friendly Frogs]{Friendly frogs, stable marriage,\\ and the magic of invariance}

\author[Deijfen]{Maria Deijfen}
\address{Maria Deijfen, Stockholm University, Sweden} \email{mia@math.su.se}
\author[Holroyd]{Alexander E.\ Holroyd}
\address{Alexander E.\ Holroyd, Microsoft Research, Redmond, USA}
\email{holroyd@microsoft.com}
\author[Martin]{James B.\ Martin}
\address{James B.\ Martin, Department of Statistics,
University of Oxford, UK}
\email{martin@stats.ox.ac.uk}

\date{11 April 2016}


\begin{document}
\maketitle

\begin{abstract}
We introduce a two-player game involving two tokens located at points of a
fixed set.  The players take turns to move a token to an unoccupied point in
such a way that the distance between the two tokens is decreased.  Optimal
strategies for this game and its variants are intimately tied to Gale-Shapley
stable marriage. We focus particularly on the case of random infinite sets,
where we use invariance, ergodicity, mass transport, and deletion-tolerance to determine game outcomes.
\end{abstract}

\section{Friendly Frogs}

Here is a simple two-player game, which we call
\df{friendly frogs}.  A pond contains several lily pads.
(Their locations form a finite set $L$ of points in
Euclidean space $\R^d$).  There are two frogs.  The first
player, \df{Alice}, chooses a lily pad and places a frog on
it. The second player, \df{Bob}, then places a second frog
on a distinct lily pad. The players then take turns to
move, starting with Alice. A move consists of jumping
either frog to another lily pad, in such a way that the
distance between the two frogs is strictly decreased, but
they are not allowed to occupy the same lily pad. (The
frogs are friends, so do not like to be moved further
apart, but a lily pad is not large enough to support them
both.) A player who cannot move \df{loses} the game (and
the other player \df{wins}). See \cref{example} for an
example game.
\begin{figure}
\centering
\includegraphics[width=.88\textwidth]{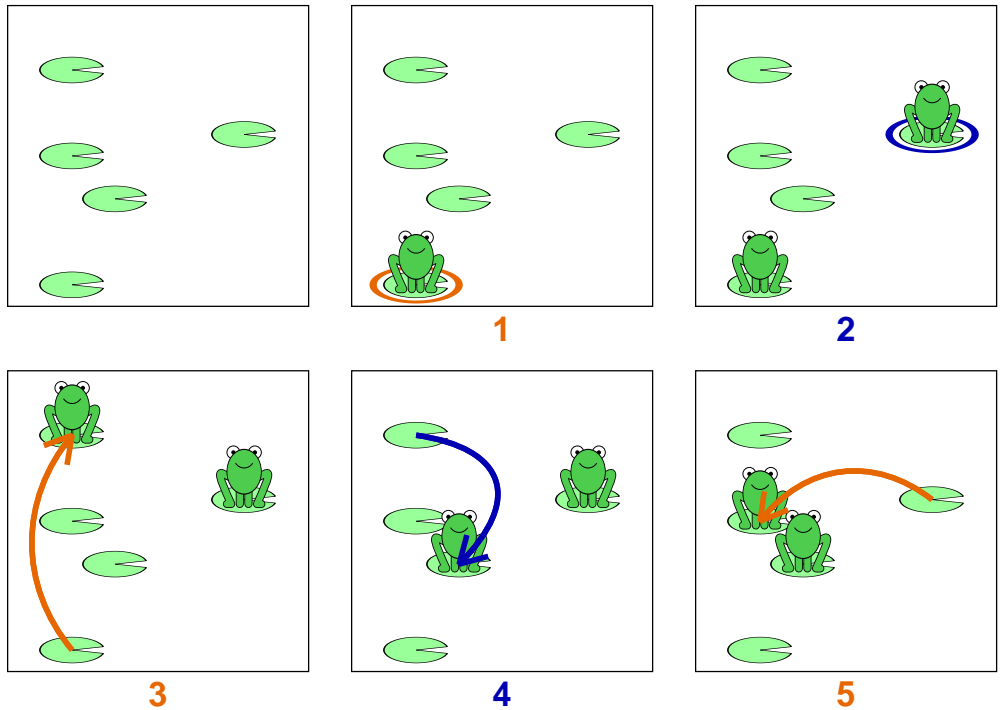}
\caption{A game of friendly frogs on a set $L$ of size $5$.
Alice starts.  Alice's moves are shown in amber, Bob's in blue.
After move $5$, Bob has no legal move, so Alice wins.}\label{example}
\end{figure}

We are interested in optimal play.  A \df{strategy} for a player is a map
that assigns a legal move (if one exists) to each position, and a
\df{winning} strategy is one that results in a win for that player whatever
strategy the other player uses.  (In friendly frogs, a position consists of
the locations of $0$, $1$ or $2$  frogs.) If there exists a winning strategy
for a player, we say that the game is a \df{win} for that player (and a
\df{loss} for the other player).

Since there are only finitely many possible positions, and the distance
between the frogs decreases on each move, the game must end after a finite number of moves.
Consequently, for any set $L$, the game is a win for exactly one player. Is it Alice or Bob?  Surprisingly,
the answer depends only on the size of $L$.

\begin{theorem}\label{theorem:finite}
Consider friendly frogs played on a finite set $L\subset\R^d$ of size $n$ in
which all pairs of points have distinct distances.  The game is a win for
Alice if $n$ is odd, and a win for Bob is $n$ is even.
\end{theorem}

\begin{proof}
Let $M$ be the set of all unordered pairs $\{x,y\}$ in $L$ such that the game
started with two frogs at $x$ and $y$ is a loss for the next player. The key
ingredient is a simple algorithm that identifies $M$.  (We postpone
consideration of the two opening moves, in which the frogs are placed). In
fact $M$ will form a partial matching on $L$. We construct this matching
iteratively as follows. The idea is to work backwards from positions where
the outcome is known. Order the set of all $\binom{n}{2}$ pairs in $L$ in
increasing order of distance between the pair.  Then for each pair in turn,
match the two points to each other if and only if neither is already matched.
The algorithm ends with at most one point not matched.  See
\cref{matching-example} for an example.
\begin{figure}
\centering
\includegraphics[width=.29\textwidth]{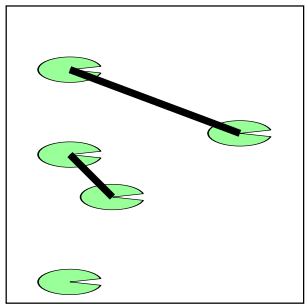}
\caption{The matching $M$ of the set $L$ in \cref{example}.}\label{matching-example}
\end{figure}

To show that this $M$ has the claimed property for the game, we need to check
that from any position in $M$, it is impossible to move to another position
in $M$, while from a position not in $M$, it is possible to move to a
position in $M$.  The former is immediate because $M$ is a partial matching
(and a move consists of moving only one frog).  For the latter, suppose the
frogs are located at $x$ and $y$, and that $x$ and $y$ are not matched to
each other.  Since $x$ and $y$ were not matched by the algorithm, at least
one of them was matched to a closer point;  without loss of generality, $x$
is matched to $w$, where $|x-w|<|x-y|$. (Here and subsequently, $|\cdot|$ is
the Euclidean norm on $\R^d$.) Therefore we can move a frog from $y$ to $w$.

If $n$ is odd then there is exactly one point that is not matched, so Alice
wins by placing the first frog there; wherever Bob places the second frog,
the two frog locations are not matched to each other.  If $n$ is even then
the matching $M$ is perfect (i.e.\ every point is matched). Therefore,
wherever Alice places the first frog, Bob wins by placing the second on its
partner in $M$.
\end{proof}

We will consider various extensions of the friendly frogs game, including
versions where frogs and/or points are player-specific (available only to one
player), where certain moves are forbidden, and where different winning
criteria apply.  Notwithstanding the humble beginning of
\cref{theorem:finite}, we will be led into some very intriguing waters. For
concreteness we will focus throughout on points in $\R^d$, although many
arguments carry over to more general metric spaces (and, for instance, the
above proof extends even to any injective symmetric distance function on
$L$).  We will continue to assume that all inter-point distances are
distinct.  (Relaxing this assumption is also quite natural, but we choose
instead to pursue other directions.)

Matters become particularly interesting when we allow the set of points (lily
pads) $L$ to be \emph{infinite}, and especially a \emph{random} countable
set.  The ``losing'' two-frog positions will still form a matching, and this
matching is most naturally interpreted as a version of the celebrated
\emph{stable marriage} of Gale and Shapley, the topic of the 2012 Nobel prize
in economics (awarded to Shapley and Roth).  We will make crucial use of
\emph{invariance} of the probability distribution of $L$ under symmetries of
$\R^d$.  This powerful tool permits remarkably simple and elegant proofs of
facts apparently not amenable to other arguments. In games involving points
of several types, we will see an example of a \emph{phase transition}, as
well as a situation in which existence of a phase transition is an open
question.  We will also analyze play of simultaneous games by making a
connection to the remarkable theory of Sprague-Grundy values (or
``nimbers'').

\enlargethispage*{1cm} The article contains a mixture of original research
and expository material. We use the friendly frogs game partly as a vehicle
to showcase some beautiful known ideas, and we assume a minimum of technical
background.  The game and its analysis are novel, so far as we know. Stable
marriage \cite{Gale} and its variants have been extensively studied, but the
connection to games appears to be new. Many of the results that we use on
matchings of random point sets are taken from \cite{Pom}.  We will review the
necessary background and give proofs where appropriate.  The general theory
of combinatorial games is highly developed (see e.g.\ \cite{winning-ways}).
We will explain the relevant parts of the theory as they apply in our
context.  Other recent work on games in random settings appears for example
in \cite{hm-tree,bhmw,hmm} and the review \cite{krivelevich}.  In a different
direction, certain games in infinite spaces have intimate connections with
general topology \cite{telgarsky}.

%
%

\section{Infinite point sets}

\cref{theorem:finite} shows that the outcome of friendly frogs on a finite
set $L$ is determined solely by the parity of the number of points (lily
pads). What happens when $L$ is infinite?  Is $\infty$ odd or even?  The
answer now depends on the choice of set; we will focus especially on the
behaviour of \emph{typical} (i.e.\ random) infinite sets.

Let $L$ be an infinite subset of $\R^d$. As before, we assume that all
distances between pairs of points in $L$ are distinct.  We call a sequence of
points $x_1,x_2,\ldots$ a \df{descending chain} if the distances
$(|x_i-x_{i+1}|)_{i\geq 1}$ form a strictly decreasing sequence. If there
exists an infinite descending chain $x_1,x_2,\dots\in L$, then it is possible
for the game to last forever. See \cref{inf-ex}. Therefore we make the additional assumption
that $L$ has no infinite descending chains. This implies in particular that
$L$ is discrete, i.e.\ any bounded set contains only finitely many points.
\begin{figure}
\centering
\includegraphics[width=.5\textwidth]{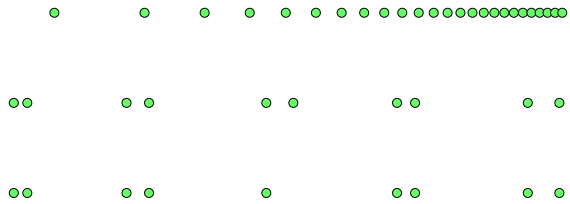}
\caption{Examples of infinite sets $L\subset \R$.
Top: play continues forever, owing to an infinite descending chain.
 Middle: Bob wins.  Bottom: Alice wins.}\label{inf-ex}
\end{figure}

It is easy to construct examples of infinite $L$ satisfying the above
conditions for which either player wins friendly frogs; see \cref{inf-ex}.
Firstly, in dimension $1$,  place exactly \textit{two} points in each of the
intervals $[3i,3i+1]$ for $i\in\Z$.  (A simple way to make all inter-point
distances distinct is to choose each point uniformly at random in the
appropriate interval, independently of all others.)  Then Bob wins by placing
a frog at the unique point in the same interval as Alice's initial frog.
Secondly, suppose the points are as above except that the interval $[0,1]$
now contains only one point. Then Alice wins by placing the first frog on
this point; whichever point Bob chooses for the second frog, Alice can then
move the first frog to the ``partner" of that point in the appropriate unit
interval.

As in the previous section, the key to analyzing the game for general $L$ is
to identify those positions from which the game is a loss for the player
whose turn it is to move.  Following standard conventions of combinatorial
game theory (see e.g.~\cite{winning-ways}), such positions are called
\df{P-positions} to indicate that the [P]revious player wins, while all other
positions are called \df{N-positions}, since the [N]ext player wins.  Since
terminal positions are P-positions, the P- and N-positions satisfy the
following.
\begin{itemize}
\item[(N)] From every N-position, there is at least one
    possible move to a P-position.
\item[(P)] From every P-position, every possible move is to an
    N-position.
\end{itemize}

Since the game terminates in a finite number of moves, it follows by
induction that these properties are sufficient to characterize the P- and
N-positions. That is, to check that a claimed partition of the positions into
P- and N-positions is correct, it suffices to check that it satisfies (N) and
(P).

In many games, characterizing the set of P-positions is a
difficult problem requiring experimentation and insight. In
contrast, checking via (N) and (P) that such a
characterization is correct may be essentially mechanical.

In friendly frogs, the two-frog P-positions are given by a matching. Here is
some notation.  Let $L\subseteq \R^d$. A \df{matching} of $L$ is a set $M$ of
unordered pairs of distinct points in $L$ such that each point of $L$ is
included in at most one pair. The matching is \df{perfect} if each point is
included in exactly one pair. For $x\in L$, we write $M(x)$ for the
\df{partner} of $x$, i.e.\ the unique point $y$ such that $\{x,y\}\in M$,
 or, if there is no such $y$, we set $M(x):=\infty$ and say
that $x$ is \df{unmatched}.

As in the case of finite $L$ in the last section, we will construct the
relevant matching iteratively. Now, however, there may be no closest pair of
points, so we need a local version of the algorithm.

The following abstraction will prove very useful. Imagine that each point of
$L$ ``prefers" to be matched to a partner that is as close as possible.
Given a matching $M$ of $L$, a pair of points $x,y\in L$ is called
\df{unstable} if they both strictly prefer each other over their own
partners, i.e.\ if $|x-M(x)|$ and $|y-M(y)|$ both strictly greater than
$|x-y|$ (where $|x-M(x)|:=\infty$ if $M(x)=\infty$, so that any partner is
preferable to being unmatched).  A matching $M$ is called \df{stable} if
there are no unstable pairs. Note that any stable matching of $L$ has at most
one unmatched point.

Stable matching can be applied to a wide variety of settings involving agents
each of which has preferences over the others.  The concept was introduced in
a celebrated paper of Gale and Shapley \cite{Gale}, who considered the
setting of $n$ heterosexual marriages between $n$ girls and $n$ boys, each of
whom has an arbitrary preference order over those of the opposite sex.  Gale
and Shapley gave a beautiful algorithm proving the existence of a stable
matching in this case. (They showed however that stable matchings are not
necessarily unique, and may not exist in the same-sex ``room-mates''
variant). As mentioned earlier, the 2012 Nobel prize in economics was awarded
on the basis of this and ensuing work, to Lloyd S.~Shapley for theoretical
advances, and to Alvin E.~Roth for practical applications.  Our setting
differs from the standard Gale-Shapley same-sex matching problem in that the
set $L$ is infinite; on the other hand, our preferences are very special,
since they are based on distance.  This case was studied in \cite{Pom}.

\begin{prop}[\cite{Pom}]\label{prop:stable_matching}
Suppose $L\subset \R^d$ has all pairwise distances distinct and has no
infinite descending chains.  Then there exists a unique stable matching of
$L$.
\end{prop}

\begin{proof}
We will show that the following algorithm leads to a stable matching. First
match all mutually closest pairs of points. Then remove them and match all
mutually closest pairs in the remaining point set.  Repeat indefinitely
(i.e.\ for a countably infinite sequence of stages), and take as the final
matching the set of all pairs that are ever matched.

By induction over the stages in the algorithm, every pair that is matched by the algorithm must be
matched in any stable matching.

Furthermore, at most one point can be left unmatched by the algorithm. To see
this, assume that there are at least two unmatched points. Since there are no
descending chains, the set of unmatched points then contains at least one
pair of points that are mutually closest in this set and, since $L$ is
discrete, this pair must have been mutually closest at some finite stage of
the algorithm. However, then they should have been matched to each other,
which is a contradiction.

Finally, we need to confirm that the resulting matching is
in fact stable.  To this end, assume that there exist
$x,y\in L$ with $|x-M(x)|$ and $|y-M(y)|$ both strictly
greater than $|x-y|$. By the previous argument, at least
one of $x$ and $y$ is matched, so consider the earliest
stage at which one of them was matched by the algorithm.
Since both $x$ and $y$ were unmatched prior to this stage,
we obtain a contradiction.
\end{proof}
\begin{figure}
\centering
\includegraphics[width=.65\textwidth]{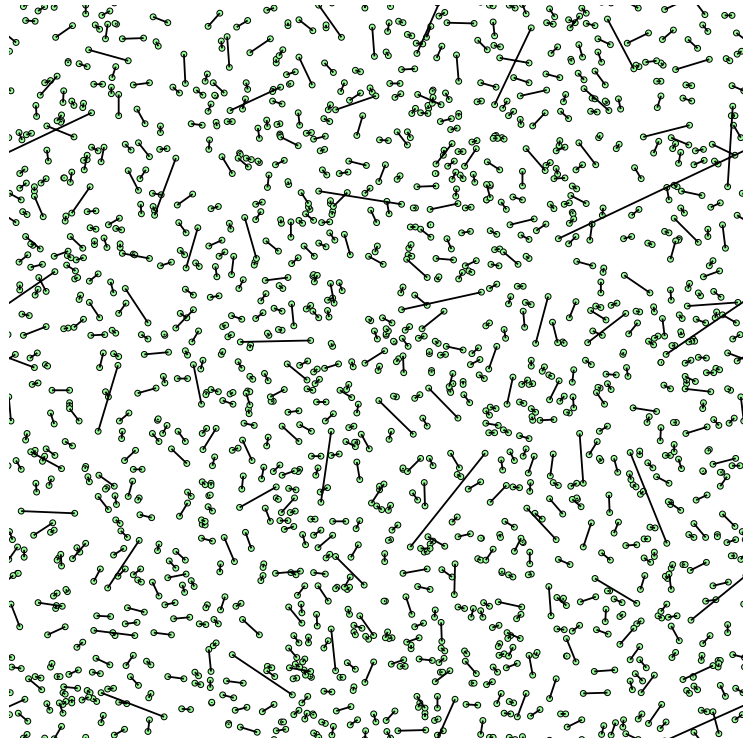}
\caption{The stable matching of random points on a two-dimensional torus.}\label{torus}
\end{figure}

\begin{samepage}
\begin{prop}\label{stable-game}
Suppose $L\subset \R^d$ has all pairwise distances distinct and has no
infinite descending chains.  Let $M$ be the stable matching of $L$ and
consider friendly frogs on $L$.  The position with the two frogs at $x$ and
$y$ is a P-position if and only if $x$ is matched to $y$ in $M$.
\end{prop}
\end{samepage}

\begin{proof}
Since $L$ has no infinite descending chains, the game terminates.  Therefore,
it suffices to check the conditions (N) and (P) above.  For (N), if
$\{x,y\}\not\in M$, then $x$ (or $y$) must have a partner that is closer than
$y$ (or respectively $x$), since otherwise $x$ and $y$ would constitute an
unstable pair. Without loss of generality, $M(x)=w$ where $|x-w|<|x-y|$, and
we can then move a frog from $y$ to $w$. The claim (P) is immediate, since
$M$ is a matching.
\end{proof}

As before, if $M$ has one unmatched point then Alice wins by placing the
first frog at that point. If the matching is perfect then Bob wins by placing
the second frog at the partner of Alice's initial move.  As we have seen,
both situations are possible for suitable infinite sets $L$.

\subsection{Random infinite sets}

It is natural to ask what happens for a \textit{typical} infinite set of
points.  A natural and canonical way to formalize this notion is the Poisson
point process, which is defined as follows.  Fix $\lambda>0$. Let any Borel
set of finite volume contain a random number of points with a Poisson
distribution of mean equal to $\lambda$ times its volume, and let disjoint
sets contain independent numbers of points.  These conditions characterize
the distribution of the set of points, and the resulting random set is called
a (homogeneous) \df{Poisson (point) process} with \df{intensity} $\lambda$ on
$\R^d$. It is a countable infinite set with probability $1$.  (The Poisson
process has other equivalent definitions -- for instance it may be
constructed as a limit as $n\to\infty$ of $n$ uniformly random points in a
ball of volume $n/\lambda$ around the origin, or as a limit as $\epsilon\to
0$ of a grid of cubes of volume $\epsilon$ each of which contains a point
with probability $\epsilon\lambda$ independently.)  If $L$ is a Poisson
process of intensity $1$ then $\{\lambda^{1/d} x:x\in L\}$ is a Poisson
process of intensity $\lambda$ -- the intensity parameter will be unimportant
for us until we consider several Poisson processes together.  See e.g.\
\cite{Daley_VJ} for background. It is straightforward to check that with
probability $1$, all pairs of points have distinct distances, and that there
are no descending chains. See e.g.\ \cite{Olle_Ronald} or \cite{daley-last}
for proofs. The process is \textbf{translation-invariant}, which is to say,
its distribution is invariant under the action of any translation of
$\mathbb{R}^d$.


\begin{theorem}\label{inf}
Let $L$ be a Poisson point process on $\R^d$.  With
probability $1$, friendly frogs on $L$ is a win for Bob.
\end{theorem}

\begin{proof}
By Proposition \ref{prop:stable_matching}, there is a unique stable matching
$M$ of $L$. It suffices to check that this matching is perfect with
probability 1. The matching has at most one unmatched point.  But if there is
an unmatched point then its location is a translation-invariant random
variable on $\R^d$, which is impossible.  More precisely, by
translation-invariance of the Poisson process and uniqueness of the stable
matching, every unit cube in $\R^d$ has equal probability $p$ of containing
an unmatched point.  We can partition $\R^d$ into unit cubes indexed by
$\Z^d$, so the probability that there exists an unmatched point is
$\sum_{z\in\Z^d} p$.  Since this sum must be finite, $p=0$, whence the sum is
$0$.
\end{proof}

Despite the simplicity of the above proof, there is something subtle and
mysterious about the argument.  What probability-one property of the Poisson
process does it use?  In other words, is there some easily described set
$\mathcal A$ of subsets of $\R^d$ such that (a) the Poisson process lies in
$\mathcal A$ with probability $1$, and (b) Bob wins on any $L\in\mathcal A$?
We do not know of such a set, except for unsatisfying choices such as
$\mathcal{A}=\{L:L\text{ has a perfect stable matching}\}$ or
$\mathcal{A}=\{L:\text{ Bob wins}\}$.  As we have seen, the set of $L$ with
distinct inter-point distances and no descending chains satisfies (a) but not
(b). The proof of \cref{inf} uses translation-invariance of the Poisson
process in a fundamental way that apparently cannot be easily reduced to such
a probability-one property. Many elegant arguments in probability theory
involve an appeal to some symmetry or invariance property of this kind.
In the next section we will use stronger probabilistic properties of Poisson
processes -- deletion-tolerance and ergodicity.

In fact, the algorithm in the proof of Proposition \ref{prop:stable_matching}
leads to a perfect stable matching for a large class of translation-invariant
point processes on $\mathbb{R}^d$ -- see \cite[Proposition 9]{Pom}. The
conclusion of Theorem \ref{inf} hence remains valid for this class of
processes.

The article \cite{Pom} is also concerned with the distribution of the
distance from a point to its partner in the stable matching.  These distances
are potentially relevant to issues of computational complexity and length of
the game.  For instance, if Alice is required to place her first frog within
distance $r$ of the origin, how difficult can she make it for Bob to win?  We
leave these interesting questions for future investigation.


\section{Colored frogs, colored points}

\begin{figure}
\centering
{\hfill
\includegraphics[width=.32\textwidth]{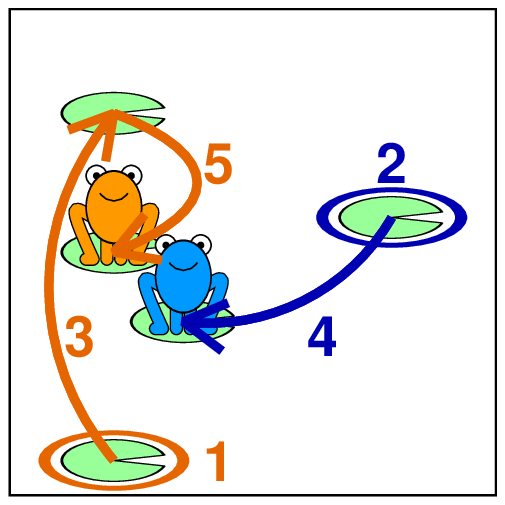}\hfill
\includegraphics[width=.32\textwidth]{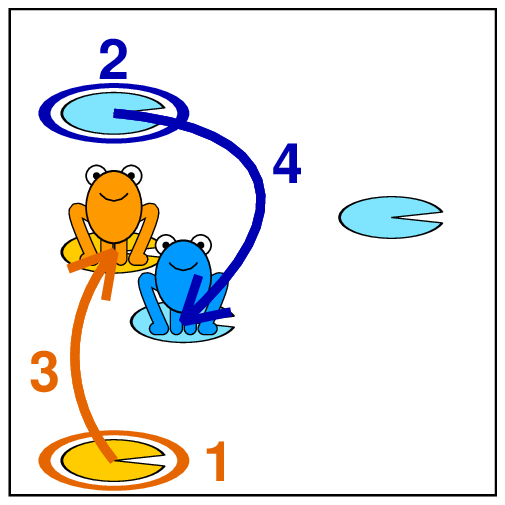}\hfill
\includegraphics[width=.32\textwidth]{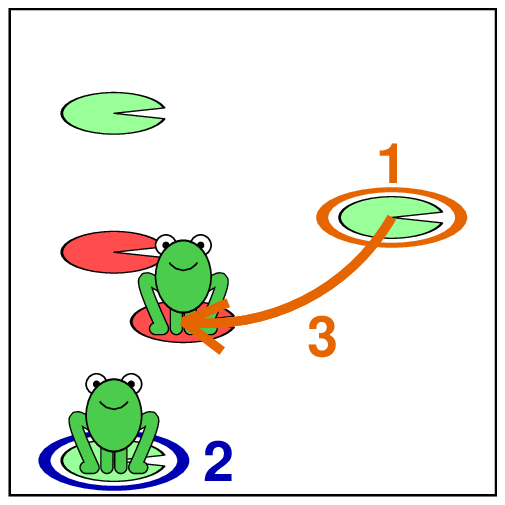}
\hfill}
\caption{Three variant games: (a) colored friendly frogs, in which each player may only move their own frog; (b) colored friendly frogs on colored points, where in addition a frog may only occupy a point of its own color; (c) fussy frogs, in which the two frogs may not both occupy red points. }\label{col-games}
\end{figure}

In this section we consider variants of friendly frogs in which frogs and/or
lily pads have multiple colors, and the allowed moves are correspondingly
restricted. Throughout we take $L$ to be an infinite set satisfying the
assumptions of Proposition \ref{prop:stable_matching}.

\subsection{Colored frogs}

First we introduce the \df{colored friendly frogs} game.
Here, Alice starts by placing an \df{amber} frog on some
point of $L$, then Bob places a \df{blue} frog on a
different point.  Subsequently, the game proceeds exactly
as before, except that Alice may only move the amber frog,
and Bob may only move the blue frog. As before, a player
who cannot move loses.

A two-frog position can now be specified by an
\emph{ordered} pair $(x,y)$, where $x$ is the location of
the frog of the previous player to move, and $y$ the
location of the frog of the next player.

Rather than requiring an entirely new analysis, it turns out that the
P-positions can again be described in terms of the stable matching $M$ of
$L$. If $|x-y|\leq |x-M(x)|$ then we say that $x$ \df{desires} $y$. (This
terminology is natural given the interpretation of preferences described
earlier.) Note the use of the weak inequality $\leq$, so that a point desires
its own partner.  Here is the analogue of Proposition \ref{stable-game} for
colored friendly frogs.

\begin{prop}\label{prop:desire}
Suppose $L\subset \R^d$ has all pairwise distances distinct and has no
infinite descending chains.  Let $M$ be the stable matching of $L$, and
consider the colored friendly frogs game on $L$. The position $(x,y)$ is a
P-position if and only if $x$ desires $y$.
\end{prop}

\begin{proof}
Again it suffices to check the conditions (N) and (P). For (N), if
$|x-y|>|x-M(x)|$, then the frog at $y$ can be moved to $M(x)$.   On the other
hand, for (P), if $|x-y|\leq |x-M(x)|$, then there cannot exist $z\in L$ with
$|x-z|<|x-y|$ and $|x-z|\leq |z-M(z)|$, since in that case $x$ and $z$ would
constitute an unstable pair. Hence moving the frog at $y$ must result in a
position $(z,x)$ with $|x-z|>|z-M(z)|$.
\end{proof}

Note that an unmatched point in the stable matching is not desired by any
other point, since that pair would be unstable. Hence, if the stable matching
of $L$ has one unmatched point, then Alice wins colored friendly frogs by
placing her amber frog at the unmatched point. If the matching is perfect,
then Bob wins, e.g.\ by placing his blue frog at the partner of Alice's
initial point. The outcome in hence the same as in the original friendly
frogs game.  In particular, we have the following.

\begin{cor}
Let $L$ be a Poisson process on $\R^d$.  With probability
$1$, colored friendly frogs on $L$ is a win for Bob.
\end{cor}

Indeed, Bob may use the same strategy in the colored and uncolored games,
always moving to a matched pair. Does this mean that the games are
essentially identical? No. To highlight an interesting difference, let us
modify the rules in a way that favors Alice.   In \df{shy} friendly frogs, we
fix a constant $c>0$, and stipulate that Bob, on his opening move, cannot
place the second frog within distance $c$ of the first frog.  (But we place
no such restriction on subsequent moves.) Shy colored friendly frogs is
defined analogously. Surprisingly, the outcome now differs between the two
variants; the proof will employ an interesting probabilistic argument.

\begin{theorem}
Let $L$ be a Poisson process on $\R^d$, and fix $c>0$. With
probability 1:
\begin{enumerate}
\item[(i)]  shy friendly frogs is a win for Alice;
\item[(ii)] shy colored friendly frogs is a win for Bob.
\end{enumerate}
\end{theorem}

\begin{proof}
For (i), Alice places the first frog on any point $x$ whose partner $M(x)$ is
at least distance $c$ away.  Such a point exists, since the stable matching
$M$ is perfect, but the Poisson process has points whose nearest neighbor is
at least distance $c$ away.

Turning to (ii), we claim that with probability $1$, every point is desired
by infinitely many others.  This implies in particular that whatever Alice's
opening move $x$, there exists a point $y$ with $|x-y|>c$ that desires $x$,
so Bob wins by placing his frog there, by \cref{prop:desire}.  The
claim follows from \cite[Theorem~1.3~(i)]{perc}.  Since the proof in our case
is short, we include it.

Let $X$ be the (random) point of $L$ closest to the origin.  It suffices to
show that infinitely many points desire $X$.  Let $D$ be the set of points
that desire $X$.  Modify the set $L$ as follows.  Whenever $D$ is finite,
delete all points of $D$ and their partners, except for $X$ itself (which is
the partner of a point in $D$).  It is easy to check that the stable matching
of the modified set is simply the restriction of $M$ to the points that
remain.  In particular, if $D$ was finite then $X$ is now unmatched. However,
the Poisson process is \df{deletion-tolerant}, which is to say: deleting any
finite set of points, even in a way that depends on the process, results in a
point process whose distribution is absolutely continuous with respect to the
original distribution. (See e.g.\ \cite[Lemma~18]{Pom} or
\cite{holroyd-soo}.) That is, the deletion cannot cause any event of zero
probability to have positive probability.  (Intuitively, the picture after
deletion is still plausible.)  Since the stable matching of the Poisson
process is perfect with probability $1$, we deduce that $D$ was infinite with
probability $1$.
\end{proof}

\subsection{Colored points}

There is a further natural variant of colored friendly frogs in which the two
frogs are restricted to different point sets. Let $\LA$ and $\LB$ be two
disjoint subsets of $\mathbb{R}^d$ whose union satisfies the assumptions of
Proposition \ref{prop:stable_matching}.  We refer to points of $\LA$ and
$\LB$ as amber and blue, respectively.  We stipulate that Alice's amber frog
can only occupy an amber point, and Bob's blue frog can only occupy a blue
point. Otherwise the rules are as for colored friendly frogs.  We call this
game \df{colored friendly frogs on colored points}. The P-positions in this
case are given by a two-color variant of stable matching.


A two-color matching of $(\LA,\LB)$ is a set $M$ of pairs of points
$(x,y)\in\LA\times\LB$ such that each point is contained in at most one pair.
As in the one-color case, the matching is perfect if each point of $\LA\cup
\LB$ is included in a pair. A two-color matching $M$ of $(\LA, \LB)$ is
\df{stable} if and only if there do not exist $x\in \LA$ and $y\in \LB$  with
$|x-M(x)|$ and $|y-M(y)|$ both strictly greater than $|x-y|$.

Proposition \ref{prop:stable_matching} and Proposition \ref{prop:desire}
remain true for this game, with $L$ replaced by $(\LA,\LB)$, ``stable
matching" replaced by ``stable two-color matching", and a revised definition
of desire under which a point can only desire a point of the opposite color
(see \cite{Pom} for more detail). The same proofs apply with only minor
adjustments. Specifically, in the algorithm described in the proof of
Proposition \ref{prop:stable_matching}, points of the same color cannot be
matched to each other. Therefore, instead of leaving at most one point
unmatched, it follows from the same arguments that all unmatched points must
be of the same color.

Note that an unmatched point desires all points of the other color, and an
unmatched point cannot be desired by any point of the other color, since they
would be an unstable pair. If the two-color stable matching has unmatched
amber points, then Alice wins by placing her frog at one of these points.  If
not, Bob wins by placing his frog on an unmatched blue point (if one exists),
or on the partner of Alice's opening move.

\begin{theorem}\label{th:2c}
Let $\LA$ and $\LB$ be two independent Poisson processes on $\mathbb{R}^d$,
with respective intensities $\alpha$ and $\beta$. Consider colored friendly
frogs with colored points on $(\LA, \LB)$. The game is a win for Bob if
$\alpha\leq\beta$, and a win for Alice if $\alpha>\beta$.
\end{theorem}

The probabilistic setup of \cref{th:2c} is equivalent to that of a single
Poisson process of intensity $\alpha+\beta$ in which each point is
independently declared amber or blue with respective probabilities
$\alpha/(\alpha+\beta)$ and $\beta/(\alpha+\beta)$.  (See e.g.\
\cite{Daley_VJ}.) The conclusion of \cref{th:2c} is an example of a
\emph{phase transition}: an abrupt qualitative change of behavior as a
parameter crosses a critical value.

To prove Theorem \ref{th:2c}, we need a property that is stronger than
translation invariance. A point process is said to be \df{ergodic} if every
event that is invariant under translations has probability 0 or 1. For
example, the event that there is no point within distance 1 of the origin is
not translation-invariant, but the event that there are infinitely many
disjoint balls of radius 1 that contain no points is translation-invariant. A
Poisson process is ergodic (and so is the two-color process made up of two
independent Poisson processes) -- this can be deduced using the independence
of the process on disjoint subsets of the space.  (See e.g.\
\cite{Daley_VJ}.)

\begin{proof}[Proof of Theorem \ref{th:2c}]
First let us consider the case $\alpha=\beta$.
The set of unmatched points in the stable matching is either empty,
or consists only of amber points or only of blue points.
Applying ergodicity, one of these three events
must have probability 1, and the others probability 0.
But by symmetry the probabilities of unmatched amber points
and of unmatched blue points must be equal.  Hence
they are both 0, and with probability 1 the matching is perfect,
giving a win for Bob.

When the two intensities are different, it is natural to expect that we
cannot match amber points to blue points in a translation-invariant way
without leaving some of the higher-intensity set unmatched.  Making this
intuition rigorous may at first appear tricky.  We might compare the numbers
of points in a large ball, but perhaps many points have their partners
outside the ball. And where should we use translation-invariance? Since $L_A$
and $L_B$ are countable infinite sets, there certainly exists \emph{some}
perfect matching between them.

In fact, there is a clean solution, using a simple but powerful tool, the
\emph{mass transport principle}.  (See \cite{blps,haggstrom97} for
background.) Consider any function $f:\Z^d\times \Z^d\to[0,\infty]$ that is
translation-invariant in the sense that $f(s,t)=f(s+u, t+u)$ for all
$s,t,u\in\Z^d$. Then note that $\sum_{t\in\Z^d} f(0,t)=\sum_{t\in\Z^d}
f(-t,0)= \sum_{s\in\Z^d} f(s,0)$.  It is sometimes helpful to think of
$f(s,t)$ as the mass sent from $s$ to $t$.

Now suppose $\alpha<\beta$. For $s\in\Z^d$, let $Q_s$ be the unit cube
$s+[0,1)^d$ in $\R^d$. Define $f(s,t)$ to be the expected number of amber
points in $Q_s$ that are matched to blue points in $Q_t$. This $f$ is
translation-invariant in the sense of the previous paragraph, because of
translation-invariance of the Poisson processes. Thus, $\sum_s f(s,0)$, which
is the expected number of matched blue points in $Q_0$, is equal to $\sum_t
f(0,t)$, which is the expected number of matched amber points in $Q_0$. The
latter is at most $\alpha$, the expected total number of amber points in
$Q_0$. But the expected number of blue points in $Q_0$ is $\beta$, so the
expected number of unmatched blue points in $Q_0$ is at least $\beta-\alpha$.
In particular, the probability that there exists an unmatched blue point is
positive. Applying ergodicity again shows that this probability is therefore
$1$. Thus Bob wins.

Similarly, if $\alpha>\beta$ then with probability 1 there are unmatched
amber points, leading to a win for Alice.
\end{proof}

Once again, the above proof uses invariance and ergodicity in a subtle and
fundamental way that cannot easily be reduced to probability $1$ properties
of the point process.  What property of $(L_A,L_B)$ guarantees Bob wins when
$\alpha=\beta$?  It is not that $L_A$ and $L_B$ have equal asymptotic
density.  Modifying the example in \cref{inf-ex}, that holds if $L_A$
consists of one point in every interval $[3i,3i+1]$ for $i\in\Z$ while $L_B$
has one point in each such interval except $[0,1]$.  But here Alice wins.

Again, the conclusion of Theorem \ref{th:2c} remains valid for a large class
of translation-invariant point processes; see \cite{Pom} for details of the
corresponding results for stable matchings.

\subsection{Fussy Frogs}

Despite the relatively complete analysis in the last two cases, we need not
go far to reach an unsolved problem.  In \df{fussy friendly frogs}, the
points again have two colors, now \df{green} and \df{red}, denoted by sets
$L$ and $L_R$ respectively.  The rules are as in the original friendly frogs
game (in particular, the two frogs are once again identical and can be moved
by either player), except that it is not permitted that \emph{both} frogs
simultaneously occupy red points.

\begin{problem*}
Let $L$ and $L_R$ be independent Poisson processes on $\R^d$ with respective
intensities $1$ and $\rho$.  Do there exist $d\geq 1$ and $\rho>0$ for which
Bob wins fussy friendly frogs with positive probability?
\end{problem*}

Fussy friendly frogs again has an associated matching, the analogue of stable
matching under the restriction that red points cannot be matched to each
other. This matching can be constructed iteratively as in the proof of
\cref{prop:stable_matching}, and Bob wins if and only if it is perfect.
Ergodicity shows that this has probability $0$ or $1$ for each $\rho$ and
$d$.  When $\rho> 1$ (and even when $\rho>1-\epsilon$ for some
$\epsilon=\epsilon(d)>0$), it is not difficult to show that there are
unmatched red points (so Alice wins); the question is whether this holds for
every positive $\rho$. This is not known for any dimension $d$, although in
\cite{AnderJamesYuval} it is proved that for any fixed $\rho>0$, there exists
$d_0=d_0(\rho)$ such that there are unmatched red points for all $d\geq d_0$.

\section{Variations on a theme}

In this section we consider some further variant
games, in which the rules are modified in more fundamental
ways.

\subsection{Playing to Lose}

We consider a \textbf{mis\`{e}re} version of friendly frogs. In general, a
game is said to be played under mis\`ere rules if the legal moves are the
same, but a player who cannot move now \emph{wins} the game instead of losing
it.  This means that a player tries to avoid moving to positions where the
next player cannot move.  Specifically, in mis\`ere friendly frogs, a player
wants to avoid having to move to a mutually closest pair.

Let $L$ satisfy the assumptions of Proposition \ref{prop:stable_matching}.
The P-positions in the mis\`{e}re game are given by a variant of the stable
matching of $L$ with the added restriction that mutually closest points
cannot be matched.  A matching $\tM$ of $L$ is said to be stable subject to
this restriction if there do not exist $x,y\in L$ that are not mutually
closest and with $|x-\tM(x)|$ and $|y-\tM(y)|$ both strictly greater than
$|x-y|$. The unique matching with this property is obtained by the following
modification of the iterative procedure used to construct the unrestricted
stable matching.  Call $x$ and $y$ potential partners of each other if they
are both unmatched and they are not mutually closest points of $L$; then
match all pairs $x$ and $y$ that are each others' mutually closest potential
partner.  Repeat indefinitely. The resulting matching has at most two
unmatched points (and if there are two such points, they must be mutually
closest points of $L$).

\begin{prop}\label{prop:misere}
Let $L\subset\mathbb{R}^d$ have distinct distances and no infinite descending
chains.  Let $\tM$ the stable matching of $L$ subject to the restriction that
mutually nearest neighbors cannot be matched. In mis\`ere friendly frogs, the
position with two frogs at $x$ and $y$ is a P-position if and only if $x$ is
matched to $y$ in $\tM$.
\end{prop}

\begin{proof}
With mis\`ere rules, all terminal positions are N-positions, and the
characterization of N-positions and P-positions is modified by replacing
condition (N) with:
\begin{itemize}
\item[(N$'$)] From every N-position that is not terminal, there is at
    least one move to a P-position.
\end{itemize}

Assume that $x$ and $y$ are not mutually closest and are not matched in $\tM$
(so that they hence define an N-position that is not terminal). If both
$|x-\tM(x)|>|x-y|$ and $|y-\tM(y)|>|x-y|$, then $x$ and $y$ would constitute
an unstable pair in $M$. Hence either the frog at $x$ could be moved to
$\tM(y)$, or the frog at $y$ could be moved to $\tM(x)$. The property (P)
follows since $\tM$ is a matching.
\end{proof}

This argument shows that the mis\'{e}re friendly frogs is a win for Alice if
and only if the restricted stable matching has exactly one unmatched point.

\begin{cor}
Let $L$ be a Poisson process on $\R^d$.  Mis\`{e}re friendly frogs is a win
for Bob with probability $1$.
\end{cor}

\begin{proof}
The argument in the proof of \cref{inf} shows that the matching $\tM$ is
perfect with probability $1$ -- it is impossible for the unmatched points to
form a non-empty finite translation-invariant random set.
\end{proof}


\subsection{Blocking and multi-matching}

The games can be modified by allowing moves to be blocked.  Consider colored
friendly frogs, but suppose that in addition to the two frogs, there are $k$
\df{stones}. After moving or placing their frog, a player then places the
stones on any $k$ points (lily pads).  The other player is then forbidden
from moving their frog to any of those $k$ points on the next move.
(Equivalently, we can imagine that the next player tries to make a move, but
the previous player can reject it and request that they try a different move,
up to $k$ times. The chess variants \emph{compromise chess} and \emph{refusal
chess} are similar; see e.g.\ \cite{wastlund}.) The rules are otherwise as in
colored friendly frogs. A player loses if they cannot move, perhaps because
all possible moves are blocked by stones. We call this game \df{$k$-stone
colored friendly frogs}.

The P-positions are related to stable multi-matchings, which were introduced
and studied in \cite{perc,deijfen-holroyd-peres}. Let $L$ be an infinite set
satisfying the assumptions of \cref{prop:stable_matching}. Let $m\geq 1$. An
\df{$m$-multi-matching} or \df{$m$-matching} is defined analogously to a
matching, except that each point may be matched to up to $m$ other points.
The $m$-matching is \df{perfect} if each point is matched to exactly $m$
points. For an $m$-matching of $L$, let $D(x)$ denote the distance to the
most distant partner of $x$, with $D(x)=\infty$ if $x$ has strictly fewer
than $m$ partners. The matching is \df{stable} if and only if there do not
exist $x,y\in L$ that are not matched to each other with $D(x)$ and $D(y)$
both strictly greater than $|x-y|$. A pair of points violating this is called
unstable.  A point $x$ \df{desires} $y$ if $|x-y|\leq D(x)$.

Proposition \ref{prop:stable_matching} extends to stable $m$-matchings. The
following modification of the iterative procedure in its proof leads to the
unique stable $m$-matching of $L$.  Call two points potential partners if
they are not already matched to each other and if neither is already matched
to $m$ other points. Match all mutually closest potential partners. Repeat
indefinitely.   See \cref{multi-pic} (left) for an example.

We remark that in the stable $m$-matching there may be more than one point
that has strictly fewer than $m$ partners, but there cannot be more than $m$
of them (otherwise there would be two that are not matched to each other).
\begin{figure}
\centering
\includegraphics[width=.495\textwidth]{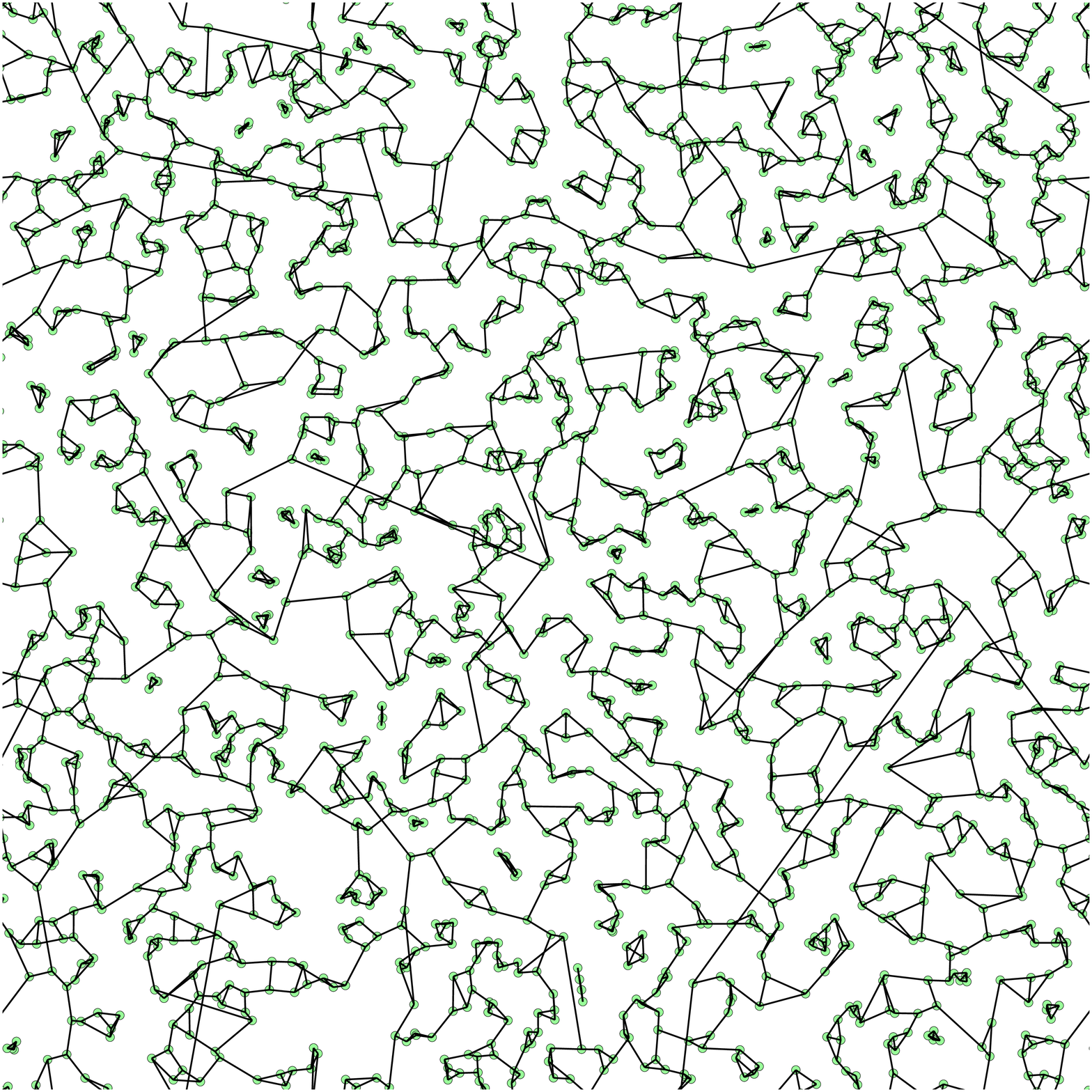}\hfill
\includegraphics[width=.495\textwidth]{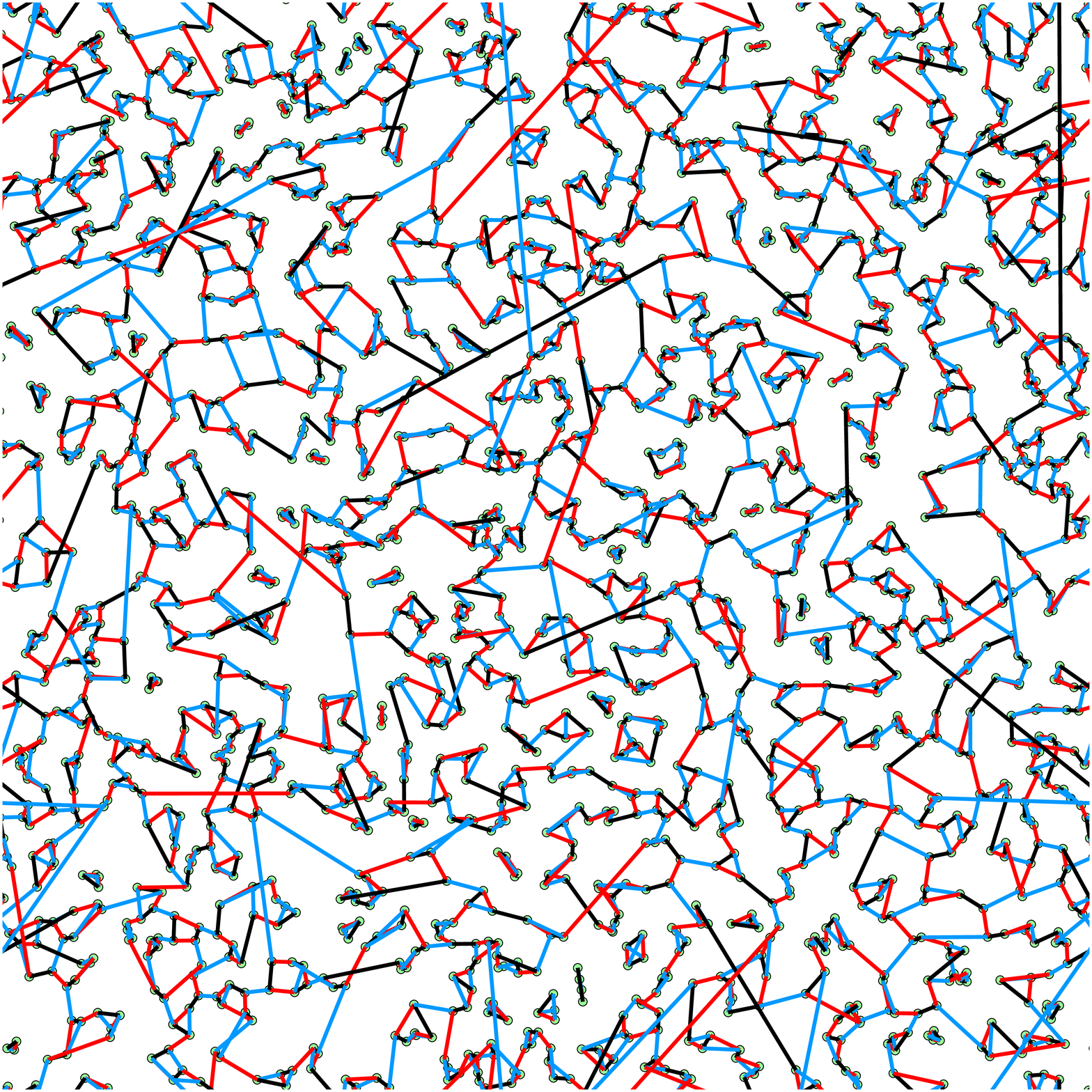}
\caption{Left: the stable $3$-multi-matching of random points in a torus.
Right: pairs having friendly frogs Sprague-Grundy values $0$ (black), $1$ (red),
 and $2$ (blue), for the same points.}\label{multi-pic}
\end{figure}

\begin{samepage}
\begin{prop}\label{prop:k} Let $L\subset\R^d$ have
distinct distances and no infinite descending chains, and assume that the
stable $m$-matching of $L$ is perfect.  Consider $k$-stone colored friendly
frogs.  Suppose that the frogs are at $x$ and $y$, with $y$ being the frog of
the next player.  This position is a P-position if and only if, in the stable
$(k+1)$-matching, $x$ desires $y$, and all partners of $x$ that are closer
than $y$ are blocked by stones.
\end{prop}
\end{samepage}

\begin{proof}
We check (N) and (P).  For (P), suppose the given conditions hold.  Since $x$
desires $y$ we have $D(x)\geq |x-y|$.  Thus, if the next player moves their
frog from $y$ to $z$, then $x$ also desires $z$.  But $z$ is not a partner of
$x$, because we assumed that all possible such $z$ are blocked.  Therefore
$z$ does not desire $x$ (otherwise they would be unstable), so the new
position is an N-position (regardless of where the player moves the stones).
We now check (N). If $x$ desires $y$ but some closer partner $z$ of $x$ is
not blocked, then the next player can move to $z$.  On the other hand, if $x$
does not desire $y$ then all the partners of $x$ are closer than $y$, and at
least one of them, $z$ say, is not blocked, so the next player moves there.
In either case, this player then blocks all $k-1$ partners of $z$ other than
$y$.
\end{proof}


\begin{theorem}
Let $L$ be a Poisson process on $\mathbb{R}^d$ and let $k\geq 1$. With
probability 1, $k$-stone colored friendly frogs on $L$ is a win for Bob.
\end{theorem}

\begin{proof}
We claim that the stable $(k+1)$-matching is perfect with probability $1$.
Indeed, there are at most $k+1$ incompletely matched points.  But the
invariance argument of \cref{inf} shows that a translation invariant random
set of points cannot have a positive finite number of points with positive
probability.

By \cref{prop:k}, Bob wins by placing his frog on an unblocked partner of
Alice's opening frog.
\end{proof}

Previous works on stable multi-matching
\cite{perc,deijfen-holroyd-peres,deijfen-lopes} have considered questions
about connectivity of the graph (many of which remain open). We do no know
whether such questions have natural game interpretations.

\subsection{Multiple Ponds and Bitwise XOR}

Finally we address how to play several games of friendly frogs
simultaneously.  Consider $k$ sets $L_1,\ldots,L_k\subset\R^d$, each assumed
to have no infinite descending chains and all distances distinct.  (We
imagine $k$ disjoint ponds, each with its own set of lily pads). In a
position of \df{$k$-pond friendly frogs}, each set $L_i$ has two frogs on two
distinct points.  (We discuss the opening moves, in which the frogs are
placed, below). Alice and Bob take turns, and a move consists of jumping one
frog in one set $L_i$ to a different point in the same set $L_i$ according to
the usual rules: the two frogs in $L_i$ must get strictly closer, but may not
occupy the same point.  A player loses if they have no legal move in any of
the sets $L_i$.

The above game is an example of a general construction; it is known as the
\emph{disjunctive sum} of $k$ copies of friendly frogs.  A remarkable theory
of such sums of games was developed independently by Sprague \cite{sprague}
and Grundy \cite{grundy}, building on Bouton's analysis of the game of Nim
\cite{bouton} (also see \cite{winning-ways} for an exposition as well as many
far-reaching extensions). It turns out that this theory fits perfectly with
friendly frogs, enabling us to show that Bob can win even with a substantial
handicap in the opening moves.

\begin{theorem}\label{multi}
Fix $k\geq 1$ and let $L_1,\ldots,L_k$ be independent Poisson processes on
$\R^d$. Consider a game of $k$-pond friendly frogs, in which Alice first
places two frogs in each of $L_1,\ldots,L_{k-1}$ and one frog in $L_k$, then
Bob places the final frog in $L_k$, and Alice moves next.  With probability
$1$, Bob wins.
\end{theorem}

In fact Bob has a unique good opening move that depends in an intricate way
on Alice's $2k-1$ initial frogs.  The key to the proof is the following
result extending stable matching to an integer-valued labeling of all pairs
of points. Write $\N:=\{0,1,2\ldots\}$. For $S\subsetneq \N$, let $\mex
S:=\min (\N\setminus S)$ be the \df{minimum excluded value}. For a set
$L\subset \R^d$ and an unordered pair of distinct points $x,y$ of $L$, let
$F(x,y)$ be the set of positions to which one can legally move in friendly
frogs, i.e.\ pairs that are strictly closer to each other than $x,y$ and
share exactly one point with $x,y$.

\begin{prop}\label{nim}
Let $L$ be a Poisson process on $\R^d$. With probability $1$, there exists a
map $G$ assigning an element of $\N$ to each unordered pair of $L$, with the
following properties.
\begin{enumerate}
\item[(i)] For every $x\in L$ and $k\in\N$ there is a
    unique $y\neq x$ such that $G(x,y)=k$.
\item[(ii)] For each pair $x,y$ we have $G(x,y)=\mex\{G(u,v): \{u,v\}\in
    F(x,y)\}$.
\end{enumerate}
\end{prop}

\begin{proof}
As before, we construct the map via an iterative algorithm.  Start with
$G(x,y)$ undefined for all $x,y$.  We say that each point of $x\in L$
\df{looks at} the closest other point $y$ for which $G(x,y)$ is currently
undefined.  For every pair $x,y$ that are looking at each other, set $G(x,y)$
to equal the smallest non-negative integer that is not currently assigned to
any pair containing $x$ or $y$.  Now repeat indefinitely.

We first check that the resulting $G$ assigns an integer to every pair of
points.  Indeed, if $G(x,y)$ is undefined then $x,y$ never looked at each
other, and so one of them, say $y$, must have a closer point $z$ for which
$G(y,z)$ is undefined.  Passing to the closest such $z$ and iterating gives
an infinite descending chain, a contradiction.

We now check the claimed properties.  For (i), it is immediate that no two
pairs containing $x$ are assigned the same integer.  It remains to check that
some pair containing $x$ has the label $k$.  Let $U_k$ be the set of points
$x$ that are not contained in any pair with label $G(x,y)=k$.  By invariance,
if $U_k$ is non-empty then it is infinite.  Let $W\subseteq U_k$ be any set
of size $k+2$. By the pigeon-hole principle there exist $u,v\in W$ with
$G(u,v)>k$. But this is a contradiction: the algorithm should instead have
assigned $u,v$ a value $\leq k$.

To check (ii), note that, during the stages of the algorithm, a given point
looks at other points of $L$ in order of increasing distance (perhaps looking
at the same point for multiple consecutive stages).  Therefore, when the
algorithm assigns a value to the pair $x,y$, all pairs in $F(x,y)$ have been
assigned values, while all other pairs that share a point with $x,y$ have
not. Therefore $G(x,y)$ is assigned the $\mex$ as claimed.
\end{proof}

It is easy to see that the set of pairs $\{x,y\}$ with $G(x,y)=0$ is
precisely the stable matching.  However, the set of pairs with $G(x,y)\leq m$
does not in general coincide with the $m$-matching considered earlier.  See
\cref{multi-pic}.  It should also be noted that the analogue of property (i)
in \cref{nim} does not hold in general for finite sets $L$ -- it is possible
that for some $x$ the set $\{G(x,y):y\in L\setminus\{ x\}\}$ is not the
interval $\{0,\ldots,L-2\}$.

\begin{proof}[Proof of \cref{multi}]
\sloppypar Let $\oplus$ denote bitwise XOR of binary expansions, so if
$a=\sum_{j\in\N} \alpha_j 2^j$ and $b=\sum_{j\in\N} \beta_j 2^j$ with
$\alpha_j,\beta_j\in\{0,1\}$ then $a\oplus b:=\sum_{j\in\N} \sigma_j 2^i$
where $\sigma_j\in\{0,1\}$ satisfies $\sigma_j\equiv \alpha_j+\beta_j \pmod
2$. Consider a position of $k$-pond friendly frogs with two frogs in each
pond, at locations $x_i,y_i\in L_i$.  We claim that it is a P-position if and
only if $\bigoplus_{i=1}^k G_i(x_i,y_i) = 0$, where $G_i$ is the map given by
\cref{nim} for $L_i$. This remarkable fact follows immediately from the
general theory (see \cite{sprague,grundy,winning-ways}), given condition
\cref{nim} (ii) on $G$ and the fact that friendly frogs is an
\emph{impartial} game (i.e.\ the same moves are available to each player) and
has no infinite lines of play. Since the proof is quite simple (given the
highly non-trivial insight of what to prove), we will summarize it below.

Given this characterization of P-positions, Bob's winning move is easy to
describe.  He computes $h:=\bigoplus_{i=1}^{k-1} G_i(x_i,y_i)$, and places
the final frog on the unique point $y_k\in L_k$ for which $G_k(x_k,y_k)=h$,
which exists by \cref{nim} (i).  Since $h\oplus h=0$, this gives a
P-position.

\enlargethispage*{1cm} Finally, we explain how to prove the claim.  As usual,
this amounts to checking conditions (N) and (P).  Let $g_i=G_i(x_i,y_i)$ and
$g=\bigoplus_{i=1}^k g_i$. For (N), suppose that $g\neq 0$. Write
$g=\sum_{j\in\N} \gamma_j 2^j$, and let $k$ be maximal such that $\gamma_k=1$
(the most significant bit of $g$). Choose $i$ such that $g_i$ also has $k$th
bit equal to $1$, and note that $g_i\oplus g<g_i$.  By \cref{nim} (ii), we
can move a frog in $L_i$ to reduce $G_i(x_i,y_i)$ to $g_i\oplus g$, resulting
in a P-position.  On the other hand, for (P), if $g=0$ then by \cref{nim}
(ii), any move changes one of the $G_i(x_i,y_i)$, giving an N-position.
\end{proof}

\bibliographystyle{abbrv}
\bibliography{frog}
\end{document}